\documentclass[journal,a4paper,twocolumn,10pt,twosided]{IEEEtran}

\usepackage{amsmath,amssymb}
\usepackage{amsthm}
\usepackage{graphicx}
 \usepackage[utf8]{inputenc}   
  \usepackage[english]{babel} 
\usepackage{color}
\usepackage{multirow}
\usepackage{hhline}
\usepackage{rotating,mathdots}
\usepackage{cite}

\usepackage{bbm}  
\usepackage{bm}  
\usepackage{array} 
\usepackage{paralist} 

\usepackage{overpic}
\usepackage{tikz}
\newcolumntype{C}[1]{>{\centering\arraybackslash}m{#1}}

\newtheorem{definition}{Definition}
\newtheorem{proposition}{Proposition}
\newtheorem{theorem}{Theorem}
\newtheorem{corrolary}{Corollary}

\newcommand{\review}[1]{\textcolor{black}{#1}}

\newcommand\RR{\mathbb{R}}
\def\R{\mathbb R}
\def\D{{\mathcal D}_{LD}}
\newcommand\NN{\mathbb{N}}
\newcommand\ZZ{\mathbb{Z}}

\def\ep{\varepsilon}
\def\la{\lambda}
\def\La{\Lambda}

\newcommand{\Dh}{\mathcal{D_H}}

\newcommand{\Dld}{\mathcal{D_{LD}}}

\newcommand{\dmin}{\mathcal{L}_{\Upsilon}}
\newcommand{\hg}{{\mathcal L}_g}
\newcommand{\hgg}{{\mathcal L}_{g_\gamma}}
\newcommand{\hgs}{{\mathcal L}_g }
\newcommand{\hq}{{\mathcal L}_Q }

\title{%
A Generalized Multifractal Formalism for the Estimation of Nonconcave Multifractal Spectra
}

\author{Roberto Leonarduzzi,~\IEEEmembership{Member,~IEEE}, 
Patrice Abry,~\IEEEmembership{Fellow,~IEEE}, %
Herwig Wendt,~\IEEEmembership{Member,~IEEE}, %
\\
Stéphane Jaffard and Hugo Touchette%
\thanks{Manuscript received February, 2018; revised September, 2018.}
\thanks{R. Leonarduzzi and P. Abry are with  Univ Lyon, ENS de Lyon, Univ Claude Bernard, CNRS,
Laboratoire de Physique, F-69342 Lyon, France ({\tt name.lastname@ens-lyon.fr}).}%
\thanks{H. Wendt is with IRIT, CNRS UMR 5505, University of Toulouse,  France ({\tt herwig.wendt@irit.fr}).}
\thanks{S. Jaffard is with  Universit\'e Paris Est, Laboratoire d'Analyse et de Math\'ematiques Appliqu\'ees, CNRS UMR 8050, UPEC,  Cr\'eteil, France, ({\tt jaffard@u-pec.fr}).}
\thanks{H. Touchette is with the National Institute of Theoretical Physics (NITheP) and the
  Department of Physics, Stellenbosch University, Stellenbosch, South Africa ({\tt htouchette@sun.ac.za}).}
\thanks{Work supported by Grant ANR-16-CE33-0020 MultiFracs.}
}

\begin{document}

\maketitle

\begin{abstract}
\boldmath
Multifractal analysis has  become a powerful  signal processing tool  that characterizes signals or images  via the fluctuations of their pointwise regularity,  quantified theoretically by the so-called \emph{multifractal spectrum}.
The practical estimation of the multifractal spectrum fundamentally relies on exploiting the scale dependence of statistical properties of appropriate multiscale quantities, such as \emph{wavelet leaders}, that can be robustly computed from discrete data.
Despite successes of multifractal analysis in various real-world applications, current estimation procedures remain essentially limited to providing \emph{concave upper-bound estimates}, while there is a priori no reason for the multifractal spectrum to be a concave function.
This work addresses this severe practical limitation and proposes a novel formalism for multifractal
analysis that enables nonconcave multifractal spectra to be estimated in a stable way. The key
contributions reside in the development and theoretical study of a generalized multifractal
formalism
to assess the multiscale statistics of wavelet leaders, and in devising a practical algorithm that permits this formalism to be applied to real-world data, allowing for the estimation of nonconcave multifractal spectra. Numerical experiments are conducted on several synthetic multifractal processes as well as on a real-world remote-sensing image and demonstrate the benefits of the proposed multifractal formalism over the state of the art.
\end{abstract}
\begin{IEEEkeywords}
\boldmath
Multifractal analysis, nonconcave multifractal spectrum, wavelet leaders, \review{Legendre transform}, generalized canonical ensemble
\end{IEEEkeywords}

\section{Introduction}
\label{sec:intro}
\subsection{Context}

Multifractal analysis is a  signal and image processing tool that permits to study a function (signal, image) $X:\RR^d\to\RR$ based on properties of its \emph{pointwise regularity}, which is quantified by a pointwise regularity exponent $h(y)$ (such as the H\"older exponent, see, e.g., \cite{Jaffard2004} and Section \ref{sec:Holder}).
More precisely, multifractal analysis characterizes the (temporal/spatial) repartition of $h(y)$
by means of the \emph{multifractal spectrum}, $ \Dh(h)$ (defined as the Hausdorff
dimension of the set of points where $h(y)=h$), that provides a global and geometric description of the \emph{fluctuations of the regularity of $X$} along $y$.

Along the last decade, multifractal analysis has become a standard statistical signal and image
  processing multiscale methodology, massively popularized by its \emph{naturally} being based on
  wavelet transforms, available in most recent and up-to-date signal/image processing toolboxes,
  cf. e.g.,
  \cite{ji2013wavelet,zhong2005image,xia2006morphology,Wendt2007,RLpl32017,Combrexelle2015}.
It has been successfully applied in various contexts for the characterization of real-world data of
different natures, ranging from natural signals (physics \cite{m74}, geophysics
\cite{telesca2011analysis, tessier1993universal}, biology  and biomedical applications
\cite{Lopes2009,soares20133d,Nakamura2016},
neurosciences
\cite{Ciuciu2009,CIUCIU:2012:A,He2010}), to man-made signals (Internet traffic
\cite{riedi1999multifractal,Abry2002,ToNet2017},
finance \cite{mandelbrot99},
art investigations \cite{ABRY:2013:B,AbrySPM2015}), to name but a few.

In practice, for discrete data, $\Dh(h)$ cannot be estimated
from its definition because neither  $h(y)$ nor  Hausdorff dimensions can be computed in a stable way.
Instead, use must be made of theoretical connections between $\Dh(h)$ and the statistics of suitable multiscale coefficients $T_X(a,y)$ computed from $X$, i.e., quantities that jointly depend on position $y$ and scale $a$.
Different $T_X$ were proposed in the literature, e.g., increments, oscillations, wavelet or multifractal-detrended fluctuation analysis coefficients 
\cite{muzy1993multifractal,kantelhardt2002multifractal}\ and, recently,  wavelet leaders
\cite{Jaffard2004,pLeadersPartI2015,pLeadersPartII2015}.
Wavelet leaders are specifically designed to meet theoretical requirements for multifractal analysis and will be used in this work. 
The choice of  $T_X$  has been studied elsewhere, e.g., \cite{Wendt2007}, and will not be further discussed here. 

The purpose of multifractal analysis is to establish a link between the multifractal spectrum $\Dh(h)$ of $X$ and the way the statistics of the quantity
$h(a,y)= \log(|T_X(a,y)|) / \log(a) $
depend on scale $a$ as $a\to0$.
Importantly, in practice, the quantities $h(a,y)$ can be robustly computed from $X$.
There exist several ways, reviewed in the next paragraph, that relate the multiscale statistics of  $h(a,y)$ to the theoretical multifractal spectrum $\Dh(h)$.

\subsection{State of the art}
\label{sec:stateoftheart}

\noindent{\bf Large deviations.} The large deviation principle and the Gartner-Ellis theorem \cite{ellis1984large} allow to establish a connection between $\Dh(h)$ and $h(a,y)$, as historically first put into light by B. Mandelbrot and collaborators (cf., e.g.,  \cite{Riedi03,Jaffard2004}).
Let $D(h,a,\epsilon)=\log\big(\textnormal{card}\{h(a,y) : h-\epsilon\leq h(a,y)\leq
h+\epsilon\}\big)/ \log a$;
the \emph{large deviation  spectrum} associated with the multiscale quantity $T_X(a,y)$ is defined as the double limit $\Dld(h)\triangleq\lim_{\epsilon\rightarrow 0}\limsup_{a\rightarrow 0} D(h,a,\epsilon)$.
It is well known that  $ \Dld$  yields an upper bound for  $\Dh$:  $\Dld(h) \geq \Dh(h)$ $\forall h$ \cite{Riedi03,Jaffard2004}, which suggests its use as an estimate of the multifractal spectrum.
It is, however, also  well documented that  taking the double limit in the definition of $\Dld$ is numerically difficult if not impossible to handle, so that the large deviation spectrum remains rarely used in practice, cf. \cite{barralgoncalves} and references therein.
One notable attempt for computing $\Dld$ was proposed in \cite{barralgoncalves}.
However, it relies on oscillations for $T_X$, and on a specific coupling of the rates of the two limits,
{which is theoretically grounded for the concave parts of  multifractal spectra  only.}

\noindent{\bf Multifractal formalism.}  As a robust alternative to the large deviation attempt, the \emph{multifractal formalism}  \cite{ParFri85,Riedi03,Jaffard2004} provides another upper-bound estimate for $\Dh$, the so-called \emph{Legendre spectrum} $\mathcal{L}$.
It is fundamentally based on the sample moments of order $q$ of $|T_X(a,y)| $, and consists in taking the Legendre transform
$  \mathcal{L}(h)\triangleq \inf_q(d+qh-\zeta(q))$ of the function
$\zeta(q) = \liminf_{a\rightarrow 0}\log \left(\frac{1}{n_a}\sum_{y=1}^{n_{a}}
|T_X(a,y)|^q\right)/\log a $.
The Legendre spectrum $\mathcal{L}$ provides a conceptually simple and numerically stable way to estimate $\Dh$, and is widely used in practice.
This multifractal formalism, using wavelet leaders as multiscale quantities $T_X(a,y)$, constitutes one of the state-of-the-art methods for estimating multifractal spectra, cf., e.g., \cite{Wendt2009,Wendt2007}.

However, as detailed in Section~\ref{sec:mff},
$\mathcal{L}$  yields a poorer bound, compared to $\Dld$, for $\Dh$:
$\mathcal{L}(h)\geq\Dld(h) \geq \Dh(h)$.
Further, the main conceptual and practical limitation of using $\mathcal{L}$ arises from the fact that it can only provide \emph{concave} estimates for $\Dld$, hence for $\Dh$.
This constitutes a strong practical restriction, since there is  a priori no reason for $\Dh$, nor for $\Dld$, to be concave functions. Thus, the practical estimation of nonconcave multifractal spectra currently constitutes one of the open challenges in multifractal analysis:
Theoretically, there are well-defined stochastic processes (such as Levy motions with Brownian components) for which $\Dh$ is known to be nonconcave;
practically, estimating nonconcave spectra could permit to detect that empirical observations actually mix several phenomena of different origins, and are hence characterized by the supremum of independent (possibly concave) spectra (cf. Section \ref{sec:results}  for examples).

\noindent{\bf Robust statistics.} While the limitations of the Legendre spectrum have long been recognized,
numerically robust procedures that allow to obtain stable nonconcave estimates for nonconcave
multifractal spectra have been proposed only recently: The quantile spectrum, $\hq$~\cite{Abel}, and
the leader profile method~\cite{esser_acha_2017}.
Both methods rely on wavelet leaders and on a modification of the large deviation spectrum
using alternative ways for studying the distributions of $h(a,y)$.
In essence, the leading idea is to replace $D(h,a,\epsilon) $ 
with robust statistics (such as quantiles).
These approaches permit to estimate the { \em increasing hull} (resp. the decreasing hull) of $\Dld$, i.e., the  least increasing (resp. decreasing) function larger than $\Dld$, hence yielding
$ \mathcal{L}(h ) \geq \hq(h)  \geq \Dld(h)   \geq \Dh(h) $.
However, they cannot work for ranges of $h$  where  $ \Dld(h)$  has a local minimum, which is at the heart of the current work.

\subsection{Goals and contributions}
The goal of the present work is to construct and study a \emph{generalized multifractal formalism} that permits the numerically stable estimation of potentially \emph{nonconcave large deviation spectra}.
The originality of the proposed solution consists in tackling the conceptual limitation of the multifractal formalism to concave estimates directly at its origin, i.e., in modifying the Legendre transform underlying the multifractal formalism (which is recalled in Section \ref{sec:mff}).
\review{The procedure is } inspired from a general ensemble thermodynamic
formalism \review{in statistical physics} (cf.
\cite{touchette2006,costeniuc2006generalized,costeniuc2005generalized,touchette2010b} and preliminary results in
\cite{LeonarduzziSSP2016}). 

The main contributions of this work are the following.
First, we develop the theoretical principle of this generalized multifractal formalism in the
context of wavelet leaders (cf. Sections \ref{Sec:glt} and \ref{sec:genmf}), and prove that it
preserves the computational advantages of the Legendre-transform-based multifractal formalism  while permitting the
estimation of nonconcave multifractal spectra, thus  yielding tight bounds for $\Dld(h)$ (hence for $\Dh(h)$). 
%
Second, we construct the corresponding practical algorithm, the \emph{generalized multifractal formalism}, that can be used to robustly compute accurate numerical estimates for multifractal spectra of general shape from finite-resolution data (cf. Section \ref{sec:GMFFest}).
Third, we  study and validate numerically the proposed generalized multifractal formalism  for
several examples of synthetic (1D) signals and (2D) images with known and prescribed multifractal
spectra, of both concave and nonconcave shapes, and compare them with the Legendre spectrum and
quantile spectrum methods (cf. Section \ref{sec:results}).
The results indicate that the proposed method yields excellent estimates for multifractal spectra, be they of concave or nonconcave shapes.
Finally, we illustrate the use of the formalism for the multifractal analysis of
a real-world satellite image.
The corresponding {\sc matlab} routines will be made available at the time of publication.

\section{Multifractal Analysis}
\label{sec:MFA}

\subsection{Pointwise regularity and multifractal spectrum}
\label{sec:Holder}

\noindent{\bf Local regularity: H\"older exponent.} Let $X(y),\;y \in \RR^d$, denote the signal, image or function to analyze.
Hereafter, we assume that $X$ is locally bounded.
This assures that pointwise H\"older regularity, used here to characterize the fluctuations of regularity in $X$, is well defined and that  {\em wavelet leaders}, the corresponding multiscale quantity, cf. \eqref{wavlead}, also are well defined (see, e.g., \cite{Wendt2009,Abel,pLeadersPartI2015} for details and ways to relax this condition).
The function $X$  belongs to the pointwise Hölder space $C^{\alpha}(y)$ if there exist $K>0$ and a polynomial $P_{y}$ (with $\makebox{deg}(P_{y})
<\alpha$) such that
$|X(y+x)-P_{y}(x)|\leq K|x|^{\alpha}$ for $x \rightarrow 0$. The \emph{Hölder exponent} of
$X$ at $y$ is defined as
\begin{equation}
\label{eq:holder}
h(y)\triangleq\sup\{\alpha\::\:X\in C^{\alpha}(y)\}
\end{equation}
and quantifies the regularity of $X$ at $y$ (see, e.g., \cite{Jaffard2004}): the smaller (larger) $h(y)$, the ``rougher'' (smoother) $X$ is at $y$.

\noindent{\bf Global description: Multifractal spectrum.} The goal of multifractal analysis is to study the repartition of the pointwise regularity  $h$ along time (or space) $y$.
Importantly, this is not achieved locally by the function $h(y)$ itself, but rather globally and geometrically via the multifractal spectrum, defined as the fractal (Hausdorff) dimension of the sets of points where the H\"older exponent takes a given value $h$
\begin{equation}
\label{equ:defDh}
 \Dh(h)\triangleq\text{dim}_{\text{H}}(\{y\in\RR^d:h(y)=h\}) \leq d,
\end{equation}
where $\text{dim}_{\text{H}}$ denotes the Hausdorff dimension (by convention,
$\text{dim}_{\text{H}}(\emptyset)=-\infty$) \review{\cite{Jaffard2004,ParFri85,Riedi03}}.

\subsection{Wavelet leaders and pointwise regularity}
\label{sec:leaders}

\noindent{\bf Wavelet coefficients.} Let $\{\psi^{(i)}\}_{i=1,\ldots,2^d-1}$ be a family of mother wavelets,
i.e., oscillating functions with fast decay, good joint time-frequency
localization, and such that the collection $\{2^{dj/2}\psi^{(i)}(2^jy-k),\;i=1,\ldots,2^d-1,\;j\in\ZZ,\;k=(k_1,\dots,k_d)\in\ZZ^d\}$  of dilated (to scale $a=2^{-j}$)
and translated versions of $\psi^{(i)}$ is an orthonormal basis of
$L^2(\RR^d)$.
The functions $\psi^{(i)}$ guarantee a number of vanishing moments $N_{\psi}\in\NN$, i.e.,
$\int y^k\psi^{(i)}(y)dy=0$ for $k=0,\cdots,N_{\psi}-1$.
The discrete wavelet coefficients of $X$ are defined as:
$c^{(i)}_{j,k}=2^{dj/2}\int_{\RR^d}X(y)\psi^{(i)}(2^{j}y-k)dy$, c.f., e.g., \cite{Mallat1998} for details.\\
\noindent{\bf Wavelet leaders.} Let $\lambda_{j,k}=\bigl[2^{-j}k_1,2^{-j}(k_1+1)\bigr)\times\cdots\times\bigl[2^{-j}k_d,2^{-j}(k_d+1)\bigr)$
denote a dyadic cube of width $2^{-j}$ at position $2^{-j}k$, and $3\lambda_{j,k}$ denote the union of $\lambda$ and its $3^d-1$ closest neighbors.
The \emph{wavelet leaders} are defined as \cite{Jaffard2004}
\begin{equation}   \label{wavlead}
L_{j,k}\triangleq \sup_{i, \lambda_{j',k'} \subset 3\lambda_{j,k} } 2^{dj/2}|c_{j',k'}^{(i)}|,
\end{equation}
that is, as the supremum of the $L^1$-normalized discrete wavelet coefficients in a narrow time neighborhood
of $y=2^{-j}k$ {for all finer scales} $j'\geq j$.\\
\noindent{\bf Wavelet leaders and pointwise regularity.}  Further, let $\lambda_{j,k_y}$ denote the only cube at scale $2^{-j}$ that includes $y$.
The local decay rate of wavelet leaders reproduces the Hölder exponent in the limit of fine scales as \cite{Jaffard2004}
\begin{equation}
\label{eq:leader_decay}
h(y) = \liminf_{j\rightarrow\infty}\frac{\log_2 L_{\lambda_{j,k_y}}}{-j}.
\end{equation}

\subsection{Wavelet leaders and spectra}
\label{sec:leadersmf}

\subsubsection{Large deviation spectrum}
\label{sec:lds}

As sketched in the introduction, the large deviation spectrum $\Dld$, here formally defined on wavelet leaders, permits to approximate the Hausdorff spectrum $\Dh$.
Motivated by \eqref{eq:leader_decay}, let us define
\begin{equation}
\label{equ:hxat}
h(y,2^{-j})=\frac{\log_2 L_{\lambda_{j,k_y}}}{-j}.
\end{equation}
Then, the large deviation spectrum is defined as
\begin{equation}
\label{equ:DLD0}
D(h,2^{-j},\epsilon)=\frac{\log_2\!\big(\textnormal{card}\{h(y,2^j) \!: h\!-\!\epsilon\leq h(y,2^j)\leq h\!+\!\epsilon\}\big)}{-j},
\end{equation}
\begin{equation}
\label{equ:DLD}
\Dld(h)\triangleq\lim_{\epsilon\rightarrow 0}\limsup_{2^{-j}\rightarrow 0} D(h,2^{-j},\epsilon).
\end{equation}
Importantly, in general, spectra $\Dh$ or $\Dld$ need not  be continuous functions.
However, while the Hausdorff spectrum $\Dh$ can be a function of any shape (and as general as an arbitrary  supremum of  a countable family of continuous functions \cite{CR-SJ-Jeunesse}), $\Dld $ is,  by construction, upper-semicontinuous.

\begin{proposition}
The large deviation spectrum $\Dld $, as defined in \eqref{equ:DLD}, is an upper-semicontinuous function on $\R \cup  \{+  \infty\} $, satisfying:
$ \forall h, h_0 \in \R, \forall \ep  >0$, $ \exists \delta >0$, if $|h-h_0| \leq \delta$, then $ \Dld(h) \leq
\Dld(h_0) + \ep$.
\end{proposition}
The proof is postponed to Appendix~\ref{sec:AppA}.

\subsubsection{Multifractal formalism and Legendre transform}
\label{sec:mff}

Also inspired from \eqref{eq:leader_decay}, the multifractal formalism is fundamentally based on the evolution across scales of the $q$th sample moments of $L_{j,k}$ (with $n_j$ the number of $L_{j,k}$ at scale $j$)
\begin{equation}
\label{sg}
S(q,j)\triangleq\frac{1}{n_j}\sum_{k=1}^{n_j}\left(L_{j,k}\right)^q,
\end{equation}
often referred-to as the \emph{structure function}.
Let the \emph{scaling function} (or \emph{scaling exponents}) be defined as
\begin{equation}
\label{equ:zetaq}
\zeta(q) \triangleq \liminf_{j\rightarrow +\infty}\frac{\log_2 S(q,j)}{-j},
\end{equation}
i.e., $S(q,j)\sim 2^{-j\zeta(q)}$, $j\rightarrow  + \infty$.
The scaling function can be related to $\Dld$, by combining the following heuristic arguments (cf, e.g., \cite{ParFri85,Riedi03,Jaffard2004}):
From (\ref{equ:DLD}), there are roughly $\sim 2^{j\Dld(h)}$ cubes $\lambda_j$ of width $2^{-j}$ which cover
locations $y$ where $h(y)=h$, and, owing to \eqref{eq:leader_decay}, each of these contributes  to $S(q,j)$ as $\sim 2^{-jq  h(y)}$. Therefore, since $n_j \sim 2^{dj}$, $S(q,j)\sim \sum_h 2^{-j(d+q h-\Dld(h))}$.
In the limit of fine scales the smallest exponent dominates, it hence follows that
\begin{equation}
\zeta(q) = \Dld^{\star}(q) \triangleq \inf_h \left\{ d+q h -\Dld(h) \right\},
\label{eq:zeta_Dh}
\end{equation}
i.e.,  $\Dld^{\star}$ is the  \emph{Legendre (or Legendre-Fenchel) transform} 
of $\Dld$, cf.,  \cite{rockafellar1997convex,Riedi03,Jaffard2004}.
The Legendre transform satisfies (cf., e.g., \cite{rockafellar1997convex, costeniuc2005generalized}):\\
\indent\emph{\underline{Property 1.}} $f^{\star}(h)$ is always a concave function of $h$.\\
\indent\emph{\underline{Property 2.}} Let $f^{\star\star} = (f^{\star})^{\star} $ 
be the double Legendre transform of $f(h)$,  thus $f^{\star\star}(x)\geq f(x)$
with equality when $f$ is concave.

Remarkably, while the $\liminf$ in \eqref{eq:leader_decay} cannot practically be replaced by a $\lim$ so that pointwise exponents cannot be robustly estimated by log-log plot regressions, it usually turns out to be the case after a space averaging as performed by $S(q,j)$. Thus, robust estimates of $\zeta$ can be practically obtained, as well as of the so-called \emph{Legendre spectrum} $ \mathcal{L}$ (using \emph{{Property 2}})
\begin{equation}
\label{equ:leg_spec}
\mathcal{L}(h) \triangleq (\Dld)^{\star\star}(h) = \zeta^{\star}(h) = \inf_q(d+qh-\zeta(q)).
\end{equation}
The Legendre spectrum $\mathcal{L}$ provides a \emph{robust} and \emph{numerically stable}
estimate of $\Dh$  because the scaling function $\zeta$ can be assessed numerically in a
robust manner, by linear regressions of $\log_2 S(q,j)$ versus $j$, cf., e.g., \cite{Wendt2007} and
Section \ref{sec:GMFFest}.
However, from Properties 1 and 2, it follows that $\mathcal{L}$ always is a concave function, regardless of the shape of $\Dld$
\begin{equation}
\label{equ:leg_spec_1}
\mathcal{L}=  \zeta^{\star} = ( \Dld)^{\star \star }  \geq   \Dld \geq  \Dh .
\end{equation}

\section{Generalized Multifractal Formalism}

\subsection{Generalized Multifractal Formalism: Principle}
\label{Sec:glt}

\subsubsection{Definition and properties}

Following \cite{touchette2006,costeniuc2006generalized,costeniuc2005generalized,touchette2010b}, the intuition underlying the construction of the proposed generalized multifractal formalism
is to \emph{lift} $\Dld $ by a known and well-chosen function $g$, to perform a
Legendre-transform-type estimate of the \emph{lifted} spectrum, and then to subtract $g$ to yield a
new, sharper bound  for $\Dld$. This bound is not necessarily concave, and the method preserves the computational advantages of the Legendre transform.

 \begin{definition}[Generalized Legendre Spectrum]
 \label{def:GLS}
 Let $g:\RR\mapsto\RR \cup \{ -\infty\}  $ be an {\em admissible} continuous function satisfying: $i)$ $g(x)\rightarrow -\infty$ when $x\rightarrow  \pm \infty$; $ii)$ $ g(x) \neq -\infty$ on  an \review{arbitrary} interval $I$ of nonempty interior; $iii)$ g is continuous on $I$.
The {generalized Legendre spectrum} is defined as:
\begin{equation}
\label{eq:def_glt}
\hg (h) \triangleq (\Dld+g)^{\star\star}(h)-g(h).
\end{equation}
\end{definition}

\review{Functions $g(x)=-x^2$ and $g(x)=-|x|$ constitute two simple examples of admissible
  functions, with conditions \emph{ii}) and \emph{iii}) being satisfied for any
  interval $I=[a,b]\subset\RR$ \cite{touchette2010b}.}

The generalized Legendre spectrum possesses the following property of key practical importance:

\begin{proposition}
\label{prop:glt}
$ \hgs \geq \Dld$, and equality holds if and only if
  $\Dld+g$ is concave.
\end{proposition}

\begin{proof}
By Property 2, $ (\Dld+g)^{\star\star}\geq  \Dld+g$, and equality holds if and only if $\Dld+g$ is concave.
Using this in (\ref{eq:def_glt}) implies $ \hgs =(\Dld+g)^{\star\star}-g \geq (\Dld+g) -g=\Dld$.
\end{proof}

From Proposition \ref{prop:glt}, $\hgs$, like  $\mathcal{L}$, yields an upper bound for $\Dld$. However, unlike $\mathcal{L}$, $\hgs$ can be \emph{nonconcave} and can hence potentially provide a better bound. Indeed,
if $ \Dld +g$ is a concave function but $\Dld$ is nonconcave, then  
$\hgs$ is nonconcave; in this case $\hgs = \Dld  $ while $\mathcal{L}\geq
\Dld$. 
The following proposition formalizes the intuition that $\hgs$  provides better upper bounds for $\Dld$ than $\mathcal{L}$.

\begin{proposition}
\label{prop:bound}
Let $g:\RR\mapsto\RR$ be a concave function. Then,
\begin{equation}
\label{eq:glt_bounds}    \forall h , \quad
\mathcal{L}(h) \geq \hgs(h)   \geq \Dld (h)  .
\end{equation}
\end{proposition}

The proof is postponed to Appendix~\ref{sec:AppB}.

\subsubsection{\review{Intuition and illustration}}
\label{sec:example}

\review{The classical Legendre spectrum $\Dld^{\star\star}$ provides the \emph{concave envelope} of $\Dld$ obtained from the set of
\emph{supporting lines} of $\Dld$ \cite{touchette2010b,costeniuc2005generalized}.
When $g_{\gamma}(h)=-\gamma h^2$, the generalized version of the Legendre transform $\hgg$ provides the \emph{parabolic envelope} of $\Dld$, obtained from the set of \emph{supporting parabolas}.
Interested readers are referred to, e.g., \cite{touchette2010b,costeniuc2005generalized} for further
details.
}

\review{To gain intuition,} the generalized Legendre spectrum and its properties are illustrated with a simple example consisting of a nonconcave large deviation spectrum, composed of two parabolas
\begin{equation}
\label{eq:parab_d}
\Dld(h) =
\begin{cases}
1 - (h+1)^2 & \quad\text{if } -2< h<0,\\
1 - (h-1)^2 & \quad\text{if } 0 \leq h \leq 2,
\end{cases}
\end{equation}
and $\Dld(h)=-\infty$ otherwise.
$\mathcal{L}$ can be computed analytically and is compared to $\Dld$ in Fig.~\ref{fig:example_double_parabolic} (top row), showing that  $\mathcal{L}(h)> \Dld(h)$ for $h\in(-1,1)$, and that $\mathcal{L} $ does not recover the nonconcave parts of $\Dld$.
With $g_\gamma(h)=-\gamma h^2$, $\gamma\geq 0$, the generalized Legendre spectrum $ \hgg  $ can also be computed analytically
\begin{equation}
\label{eq:parab_gen_d}
\hgg(h)\!\!=\begin{cases}
1\!-\!(h\!+\!1)^2 & \!-2\!<\!h\! <\! -\frac{1}{1 +\gamma},\\
\frac{1}{1+\gamma}+\gamma h^2 & |h| \leq \frac{1}{1+\gamma},\\
1\!-\!(h\!-\!1)^2 & \frac{1}{1 +\gamma} \!<\! h \!<\! 2.\!\!\!\!\!\!\!\!\!\!\!\!
\end{cases}\!\!\!\!
\end{equation}
Fig.~\ref{fig:example_double_parabolic} (bottom left) compares $\Dld(h) + g_\gamma(h) $ (dotted lines) and
$(\Dld + g_\gamma)^{\star\star}(h)$ (solid lines), for several values of $\gamma$.
\review{It shows that the difference between $\Dld + g_\gamma$ and $(\Dld+g_\gamma)^{\star\star}$
(e.g., the area of the nonconcave region) decreases as $\gamma$ increases, and thus that the double
Legendre transform provides increasingly better estimates.
Indeed, the
}
corresponding generalized Legendre spectra  $\hgg$  (bottom right panel)
show that $\hgg(h)=\Dld(h)$ for $ h$ outside of the interval $h\in(-\frac{1}{1+\gamma},\frac{1}{1+\gamma})$.
Within that interval, $\mathcal{L}(h)>\hgg(h)>\Dld(h)$, hence for any $ \gamma>0$, $\hgg$ provides a
more accurate bound for $\Dld$ than  $\mathcal{L}(h) $ does.
Further, one observes that when $\gamma\rightarrow\infty$, $\hgg(h)  \rightarrow \Dld(h)$, thus showing that tuning $\gamma$ permits to achieve arbitrarily sharp bounds.

\begin{figure}
\centering
\includegraphics[width=\linewidth,trim={1mm 3mm 0mm 0mm},clip]{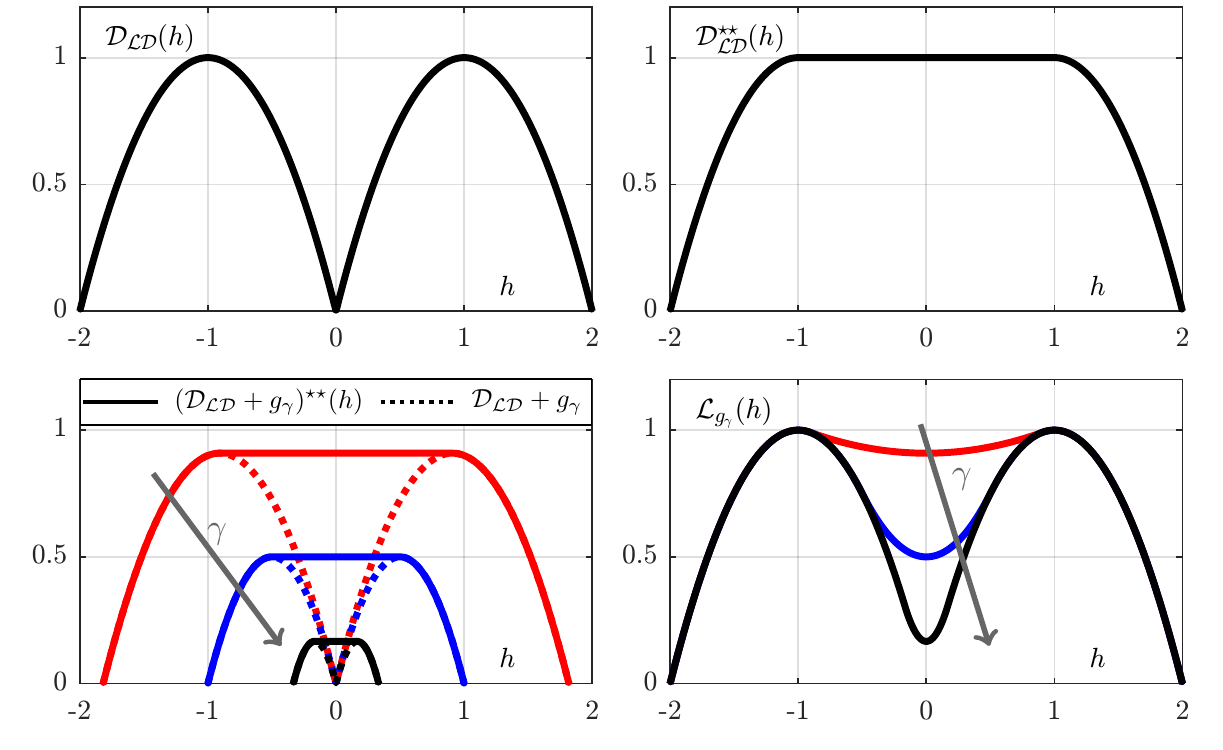}
\caption{{\bf Generalized Legendre spectra.} Top left: Large deviation spectrum $\Dld(h)$, defined as a double parabolic function.
  Top right:
  $\mathcal{L}(h)=\Dld^{\star\star}(h)$.
  Bottom left: For several values of $\gamma$,
  $\Dld(h) + g_\gamma(h)$ (dashed lines) and 
  $(\Dld(h) + g_\gamma)^{\star\star}$ (solid lines).
  Bottom right: Generalized Legendre spectra
$\hgg(h)$.}
\label{fig:example_double_parabolic}
\end{figure}

\subsubsection{Extension to a complete family of functions $g$}

The example above suggests that $\hgg $ can provide better bounds for $\Dld$ by using several functions $g$.
This intuition is formalized in the following key theoretical result, that shows that practically any $\Dld$ can be recovered from the minimum  of $\hgg$ obtained from a (possibly large) collection of dilated and translated templates of $ g $.

\begin{theorem}
\label{Theorem:global}
Let $g$ be an admissible function.
Let $\{g_{\gamma,\delta}(x) = g(\gamma(x-\delta))\}_{(\gamma,\delta)\in(\RR^+,\RR)}$ denote the collection of all dilated and translated templates of $g$.
For any dense countable set
$\Upsilon \subset \RR^+\times\RR$, and $\forall (\gamma,\delta)\in\Upsilon$,
\begin{equation}
\label{eq:dh_min}
{\mathcal L}_{g_{\gamma,\delta}}\review{(h)} \geq  \dmin\review{(h)} \triangleq \inf_{(\gamma,\delta)\in\Upsilon} \left\{ \mathcal{L}_{g_{\gamma,\delta}}\review{(h)} \right\} = \Dld\review{(h)}.
\end{equation}
\end{theorem}
The proof is detailed in Appendix~\ref{sec:AppC}.
\review{In (\ref{eq:dh_min}), the $\inf$, for different values of $h$, may be attained for different $(\gamma,\delta)$, an outcome of major importance in practice as further discussed in Section~\ref{sec:tuning}.}

\subsection{Generalized  Multifractal Formalism: Definition}
\label{sec:genmf}
It will now be explained how the classical multifractal formalism (recalled in Section \ref{sec:mff}) must be modified to permit the actual estimation from data of the generalized Legendre spectrum  $\hg$.

First, new multiscale quantities, the \emph{generalized wavelet leaders}, need to be defined  for measuring  $\hg$, as the counterpart of the wavelet leaders for measuring $\mathcal{L}$.

\begin{definition}[Generalized Wavelet Leaders]
For an admissible function $g$ and $ q \in \RR$, the {generalized wavelet leaders}
$L^{{(q,g)}}_{j,k} $ are defined %
 as follows.
Let $c_{1,0} $ be defined from the linear regression
\begin{equation}
\label{eq:def_c1}
\frac{1}{n_j}\sum_{k=1}^{n_j} \log_2 L_{j,k} = c_{1,0} + c_1 j,
\end{equation}
\begin{equation}
\label{eq:def_phi}
\makebox{ then } \phi_{j,k}\triangleq\frac{\log_2 (L_{j,k}) -  c_{1,0}}{-j},
\end{equation}
\begin{equation}
\label{eq:defp}
\makebox{ and }  L^{{(q,g)}}_{j,k} \triangleq
2^{-j\left(q\phi_{j,k} -  g\left(\phi_{j,k}\right)\right)}.
\end{equation}
\end{definition}
The occurrence of $c_{1,0}$ stems from interpreting \eqref{eq:leader_decay} as stating that $ L_{\lambda_{j,k_y}} \sim \kappa(y)2^{-jh(y)}$, and replacing pointwise estimates by averages.

Second, from these $  L^{{(q,g)}}_{j,k}  $, a generalized multifractal formalism is devised and shown to yield a tight upper bound for $\Dld$, the key theoretical contribution of this work.
Let $S_{g}(q, j)$ denote the generalized structure functions,
\begin{equation}
\label{eq:sf_gamma}
S_{g}(q, j) = \frac{1}{n_j} \sum_{k=1}^{n_j} L_{j,k}^{(q,g)} 
\end{equation}
and let the generalized scaling exponents $\zeta_g(q) $ be defined as
\begin{equation}
\label{eq:segamma}
\zeta_g(q) =  \liminf_{j\rightarrow +\infty}\frac{\log_2S_g(q,j)}{-j}.
\end{equation}
\begin{theorem}[Generalized Multifractal Formalism]
\label{theorem:gmff}
The generalized Legendre spectrum $\hg$
can be computed
as the Legendre transform of  $\zeta_g$
\begin{equation}
\label{eq:gmf}
\hg(h)=\zeta_g^{\star}(h)-g(h).
\end{equation}
\end{theorem}

The proof is detailed in Appendix~\ref{sec:AppD}.
Heuristically, the argumentation follows the intuition yielding the classical multifractal formalism (Section~\ref{sec:mff} (\ref{eq:zeta_Dh}-\ref{equ:leg_spec_1})):
Each location $y$ where $h(y)=h$ contributes to $S_{g}(q,j)$  as $\sim 2^{q  h(y)+ g( h(y))}$.
Therefore, $S_{g}(q,j)\sim 2^{j(d+q h - g( h)-\Dh(h))}$.
In the limit of fine scales,
$\zeta_{g}(q) = \inf_h \left\{ d+q h - g( h) -\Dh(h) \right\}  $ 
and therefore
$\hg(h) = \inf_q\left\{d+qh-\zeta_g(q)\right\} - g(h)$.

A major consequence of Theorem~\ref{theorem:gmff} is a sequence of inequalities in the bounds for the multifractal spectrum.
\begin{corrolary}
For any dense countable set  $\Upsilon \subset \RR^+\times\RR$,
with $g$ an admissible function,
\begin{equation}
\label{eq:ineq}
\mathcal{L} \geq {\mathcal L}_{\cal Q} \geq \dmin  = \Dld \geq  \Dh.
\end{equation}
\end{corrolary}

\subsection{Generalized Multifractal Formalism: Computation}
\label{sec:GMFFest}
\subsubsection{Estimation for finite-resolution data}
The generalized multifractal formalism can be computed for finite-resolution, discrete data $X$ 
as follows.
First, the wavelet coefficients and  leaders of $X$ are computed for each scale from \eqref{wavlead}.
Second, the generalized wavelet leaders $L_{j,k}^{(q,g)}$ and structure functions are computed using \eqref{eq:def_c1} to \eqref{eq:sf_gamma}, for each scale $j$, for a range of positive and negative values for $q$.
Third, because \eqref{eq:segamma} essentially means that  $S_{g}(q, j) \sim K_q 2^{-j\zeta_g(q)}$, the exponents $\zeta_g(q)$ are estimated by linear regressions of $\log2\big(S_{g}(q, j)\big)$ versus scales $j$
\begin{equation}
\label{eq:zetahat}
\hat\zeta_g(q)=\sum_{j_1\leq j \leq j_2} w_j \log_2\big(S_{g}(q, j)\big)
\end{equation}
where $j_1$, $j_2$ delimit the range of scales where the regression is performed, and $w_j$ are suitable linear regression weights, cf. \cite{Wendt2007,Wendt2009} and references therein for details.
Finally, applying the  Legendre transform as in \eqref{eq:gmf} provides $\hg$. 
As an alternative to the direct numerical calculation of the Legendre transform, 
$\hg$ can equivalently be obtained by a parametric formulation $\big(h(q), \hg(h(q))\big)$ similar to the original proposition in \cite{c89,Wendt2007}.

\subsubsection{Choosing the function $g$}
\label{semicont}

\review{In choosing $g$, the only fundamental requirement lies in its being an admissible function (cf. Definition~\ref{def:GLS}).
Further, Propositions~\ref{prop:glt} and \ref{prop:bound} advocate for the use of concave functions, therefore leaving a large freedom.
In principle, for each $\Dld $ with particular departures from concavity, there might exist a theoretically optimal $g$.
For instance, the example in Section~\ref{sec:example} may suggest that $ g(h) =  -\gamma |h| $ is optimal as $\Dld $ contains a nondifferentiable point.
However, adjusting $g$ to an unknown $ \Dld$ would require the design of a complex adaptive/iterative  strategy.}
\review{Instead,
in \cite{costeniuc2005generalized} (Theorem 5.2), it is proven that $ g(h) = -\gamma h^2$  can recover any nonconcave spectrum $\Dld(h)$ as long as the parameter $\gamma$ is large enough: $\gamma\geq\sup_h\Dld''(h)$.}
\review{Even though such a theoretical optimality relies on twice differentiable spectra, in practice, numerical simulations  reported in Sec. \ref{sec:results_1d} show that even for nonconcave spectra that are locally nondifferentiable, the generic choice  $ g(h) = -\gamma h^2$ remains as good as any ad hoc choice, thanks to the joint use of several different $\gamma$. }

\review{Therefore, aiming to propose a generic procedure that works for any a priori unknown $ \Dld$,
we promote for real-world applications the generic use of tunable collection of functions
\begin{equation}
\label{eq:g}
g_{\gamma,\delta}(h)=-\gamma(h-\delta)^2,\quad\gamma\geq 0,
\end{equation}
parametrized by the curvature and shift parameters $\gamma$ and $\delta$, respectively.
Section~\ref{sec:varyg} further comforts that this constitutes a versatile and generic enough choice.}

\subsubsection{Parameter tuning}
\label{sec:tuning}

The practical use of the generalized multifractal formalism proposed in Sections~\ref{sec:genmf} and \ref{sec:GMFFest} with the choice in \eqref{eq:g} for $g$ implies the selection of three parameters $q$, $\gamma$ and $\delta$, which can not be tuned independently.
In principle, a large range of values of $\gamma$ (ideally, up to $\gamma\rightarrow\infty$) are needed to allow for
multifractal spectra of any nonconcave shape to be estimated.
Also in principle, a large set of values of $q$, both positive and negative, is needed to recover the spectrum on all of its support (cf., e.g., \cite{Jaffard2004,Wendt2007}).
\review{In practice, however, large values for both $\gamma$ and $|q|$ give rise
  to a well-known
  numerical issue in multifractal analysis: the so-called linearization effect, extensively studied
  in e.g. \cite{m74,lashermes_new_2004}.
  In an nutshell, large values of $|q|$ and/or $\gamma$ cause the sum in
  (\ref{eq:sf_gamma}) to be dominated by the largest $L_{j,k}$ and hence to be heavily biased, cf. \cite{lashermes_new_2004}.
}
This is illustrated in Fig.~\ref{fig:effect_params} (left), showing that estimates for large values of $\gamma$ become increasingly biased towards the limits of the support of the spectrum (with large values $|q|$).
The numerical issues can be mitigated by restricting $q$ to a narrow range of smaller values. To compensate for the resulting restriction to a smaller range of $h$,
$g_{\gamma,\delta}$ is shifted using several values of $\delta$, i.e.,  $g_{\gamma,\delta}$ sweeps over
the full support of the spectrum, as illustrated in Fig.~\ref{fig:effect_params} (right).
Thus, a family of spectra $\mathcal{L}_{g_{\gamma,\delta}}$ is obtained, and their infimum is taken to produce
the final estimates  $\hg$ in \eqref{eq:dh_min}.

\begin{figure}
  \centering
  \includegraphics[scale=1]{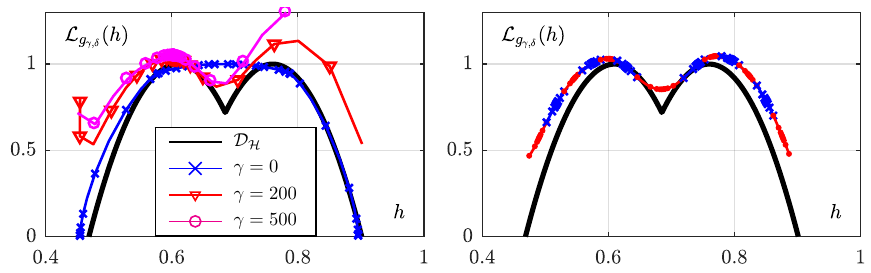}
\caption{\textbf{\boldmath Generalized Legendre spectrum ${\cal L}_{g_{\gamma,\delta}}$ estimation and parameter $\gamma$, $\delta$ and $q$ tuning.}
Left: Using a large range of $q$, $q\in\{-60,60\}$, and a few values of  $\gamma$ yields numerically unstable estimates of ${\cal L}_{g_{\gamma,\delta}}(h)$ (for illustration, $\delta=0.65$ is fixed);
Right: using a restricted range of $q\in[-4,4]$ and a large collection of $\delta\in[0.48,0.88]$ (marked by alternating colors and markers) yields far more stable and very satisfactory estimates of  of ${\cal L}_{g_{\gamma,\delta}}(h)$
(for illustration, $\gamma=200$ is fixed).
}
\label{fig:effect_params}
\end{figure}

\section{Performance assessment}
\label{sec:results}

\review{To assess the relevance and statistical efficiency of the proposed generalized multifractal formalism, several synthetic multifractal (1D or 2D) processes 
are analyzed.
Performance is assessed both for single sample paths and statistically as averages across Monte Carlo realizations.
The proposed generalized Legendre spectra $\hg$ are compared against the conventional Legendre spectrum estimates $\mathcal{L}$, and against the quantile spectrum $\hq$, mentioned in Section~\ref{sec:stateoftheart} and studied in \cite{Abel,esser_acha_2017}.}

\subsection{Synthetic multifractal processes and functions}

\review{Several representative examples are used to cover a large variety of multifractal stochastic processes and deterministic functions.
Levy stable processes with a Brownian component and Thresholded Deterministic Wavelet Cascade are chosen as representative
examples 
amongst the very few cases having proven and theoretically studied nonconcave multifractal spectra.
Such examples are very different in nature: the former is an additive selfsimilar processes, while
the latter is constructed from multiplicative cascades, the generic multifractal paradigm.
Further,  concatenations of processes having different multifractal spectra are used as  potentially generic models for real-world data with possible changes of states.
Notably, these provide interesting models for mixtures of textures in images (2D fields).
Finally, the classical multifractal random walk, with concave multifractal spectra, is used to show that the proposed procedure does not create spurious nonconcavity in the estimates.}

\paragraph{Lévy-stable process with additional Brownian component}
Lévy-stable processes are stochastic processes, whose increments are independent and stationary, and have stable distributions, governed by parameter $\alpha, \;\alpha\in(0,2)$, cf., e.g.,  \cite{Samorod1994}.
A Brownian component is added to produce a more general class of Lévy processes, whose multifractal spectrum, theoretically established in \cite{Jaffard1999}, is in general nonconcave and noncontinuous
\begin{equation}
\label{eq:levy}
\Dh(h) =
\begin{cases}
  \alpha h &\text{if }h\in[0,1/2)\\
  1 &\text{if }h=1/2\\
  -\infty & \text{otherwise.}
\end{cases}
\end{equation}

\paragraph{Thresholded Deterministic Wavelet Cascade (DWC)}
Let $w\in[0,1]$. A binomial DWC $f$ is obtained by prescribing its wavelet coefficients as $c_{0,1}=1$, $c_{j,2k}=wc_{j-1,k}$, and $c_{j,2k+1}=(1-w)c_{j-1,k}$.
Its multifractal spectrum is given by
\begin{equation}
\label{eq:3}
\Dh(h)=(1-\alpha(h))\log_2(1-\alpha(h)) - \alpha(h)\log_2\alpha(h),
\end{equation}
where $\alpha(h)=(h+\log_2(1-w))/(\log_2(1-w)-\log_2w)$ \cite{m74}. 
The thresholded-DWC $\breve f^{\theta}$ is defined by  applying a hard threshold $2^{-\theta j}$ to the wavelet coefficients of the DWC $f$:
$\breve c^{\theta}_{j,k}=c_{j,k}\mathbf{1}_{|c_{j,k}|\geq 2^{-\theta j}}$ \cite{Seuret2006}.
The multifractal spectra $\Dh(h)$
and $\breve \Dh^{\theta}(h)$ of $f$ and $\breve f^{\theta}$ are related as $\breve \Dh^{\theta}=(\sup(\Dh(h), \Dh(\Omega^{-1}(h)))$, with
$\omega$ the increasing
function $u \in [\theta, h_{max}] \rightarrow \theta (u + \log(1-w))/(\theta + \log(1-w)) \in (0,\infty)$
\cite[Theorem 5.2]{Seuret2006}.
A thresholded-DWC  is thus characterized by a nonconcave multifractal spectrum.

\paragraph{Concatenation of Multifractal Random Walk (MRW)}
MRW is defined from two independent Gaussian processes, with specific covariance structures chosen to mimic that of multiplicative cascades.
The multifractal spectrum of d-dimensional MRW consists of a parabola: $ \Dh(h)= d - (h-(H+\lambda^2/2))^2/2\lambda^2$ \cite{bacry2001}.
To simulate the situation where a nonconcave multifractal spectrum arises as a result
of nonstationarity, we concatenate two MRW with different multifractal spectra $\Dh_1$ and $\Dh_2$, yielding a  multifractal spectrum
 $\Dh(h)=\sup(\Dh_1(h),\Dh_2(h))$.

\subsection{Monte Carlo simulation set-up}

For each random process, $N_{MC}=100$ independent realizations are analyzed, with sample sizes of $N=2^{20}$ (for 1D processes) and $N_1\times N_2=2^{10}\times 2^{10}$ (for 2D processes). 
The synthesis parameters for the processes were set to: \emph{a)} $\alpha=1.25$ for the Lévy process, \emph{b)} $H_1=0.6$,
$H_2=0.75$, and $\lambda_1=\lambda_2=\sqrt{0.01}$ for the concatenation of MRW, and \emph{c)} $w=0.45$ for the binomial wavelet cascade.
A Daubechies wavelet with $N_{\psi}=3$ vanishing moments was used, and scaling exponents were computed using unweighted linear regressions over a range of scales that is hand-tuned for each example. The parameters $\gamma$ and $\delta$ were set
to $\gamma \in \{0,5,10,100,200,500\}$ and $\delta \in [c_1-0.3,c_1+0.3]$ (with $c_1$ as defined in \eqref{eq:def_c1} corresponding to the mode of $\mathcal{L}$, cf. \cite{Wendt2009}).
For the computation of $\hq$, following the method in \cite{Abel}, the quantiles for $\Dh$ and $h$ were computed over $600$ and $60$ intervals, respectively.

\subsection{Estimated (non necessarily concave) spectra}

\subsubsection{Logscale diagrams}
Fig.~\ref{fig:logscale} displays log-log plots of $S_{g_{\gamma,\delta}}(q, j)$, for several values of $q$, $\gamma$, and $\delta$ (for concatenated MRW).
It shows power-law decays of the $S_{g_{\gamma,\delta}}(q, j)$ across a wide range of scales and for any sets of parameters, as postulated in \eqref{eq:segamma}.
These empirical observations constitute crucial results that, per se, validate the practical applicability of the generalized multifractal formalism.
Notably, it permits the robust estimation of the generalized scaling exponents $\zeta_{g_{\gamma,\delta}}(q)$ from $\log_2 S_{g_{\gamma,\delta}}(q, j)$ by  linear regressions \eqref{eq:zetahat}.
\review{Similar plots are obtained for all tested processes.}

\begin{figure}
  \centering
  \includegraphics[scale=1]{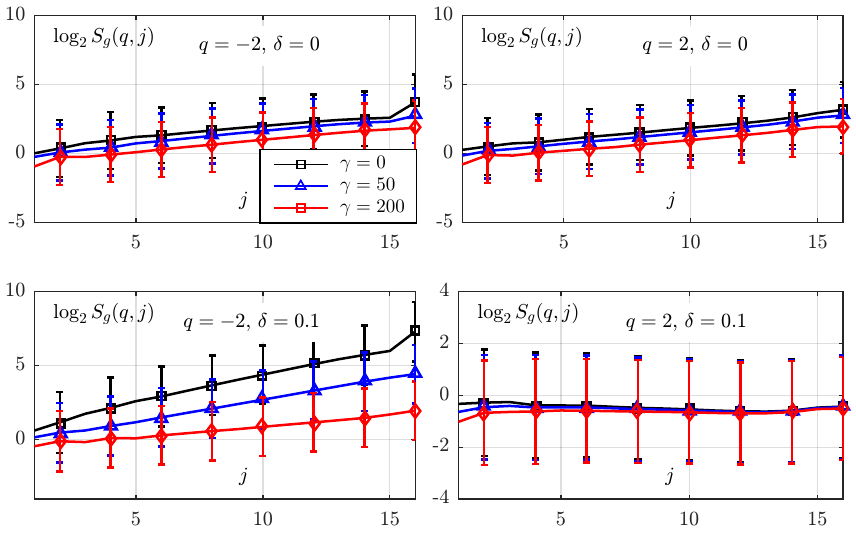}
\caption{\textbf{Logscale diagrams.} Plots of $\log_2 S_{g_{\gamma,\delta}}(q, j)$ vs. $\log_2 2^j$
  show satisfactory linear (hence scale-free) behaviors for any
  choice of parameters $\gamma$, $\delta$ and $q$ (averages over $100$ realizations and \review{95\%
  confidence intervals}).
}
\label{fig:logscale}
\end{figure}

\subsubsection{1D processes: single realization and ensemble estimation performance}
\label{sec:results_1d}

\review{Fig.~\ref{fig:results1d_single} (right column) reports $\dmin$ estimated from a single
  realization (left column) for each of the four examples.
Fig.~\ref{fig:results1d_mc} (left column) reports, for the three stochastic processes used here, estimates of
$\dmin$ averaged across independent realizations, as well as confidence intervals.
In both cases, estimates $\dmin$ are compared to the theoretical
multifractal spectra $\Dh$, to the estimated Legendre $\mathcal{L}$ and quantile $\hq$ spectra.}

\review{Fig.~\ref{fig:results1d_single} and Fig.~\ref{fig:results1d_mc} both show that nonconcavity is well evidenced by the proposed generalized multifractal spectra $\dmin$, on average as well as when applied to a single realization of data.
Further, for processes with concave spectra, estimates $\dmin$ are in excellent agreement with conventional
Legendre spectra $\mathcal{L}$, clearly showing that $\dmin$ does not introduce spurious nonconcave
regions. This is the case even on single realizations, suggesting that the absence of spurious
nonconcave regions is not a product of the smoothing performed by the ensemble average.}

\review{Fig.~\ref{fig:results1d_single} and Fig.~\ref{fig:results1d_mc} also show that the estimates yielded by $\dmin$ outperform the quantile-based $\hq$.
Focusing on the first example (top row), one can see that $\dmin$
provides estimates across the full support of the multifractal
spectrum while $\hq$ does not, which can be explained by the small
number of wavelet leaders actually corresponding to the missing $h$,
see also \cite[Fig. 9]{esser_acha_2017}.}
\review{Further, the estimates $\dmin$ are overall of better quality than those of $\hq$ for the
  nonconcave parts of $\Dh$. First, for the thresholded-DWC (Fig.~\ref{fig:results1d_single} second
  row), they provide a better estimation of the weakly nonconcave part on the
  decreasing branch caused by the thresholding. Second, they show a clear evidence of nonconcavity for the
  concatenated MRW (Fig.~\ref{fig:results1d_single} third row), while $\hq$  completely misses the
  nonconcavity, as expected for processes   whose spectra contain several local maxima.}

\review{At the statistical performance level, confidence intervals, being extremely narrow, see
  Fig.~\ref{fig:results1d_mc}, suggest overall low estimation variances.
Further, root mean-squared errors (rmse) of estimates for $\dmin$,  $\mathcal{L}$ and $\hq$ are
 compared in Fig.~\ref{fig:results_1d_estperf}, showing that $\dmin$ globally yields the lowest
rmse, notably in the nonconcave regions, where it has the lowest bias. }

\begin{figure}
  \centering
  \includegraphics[scale=1]{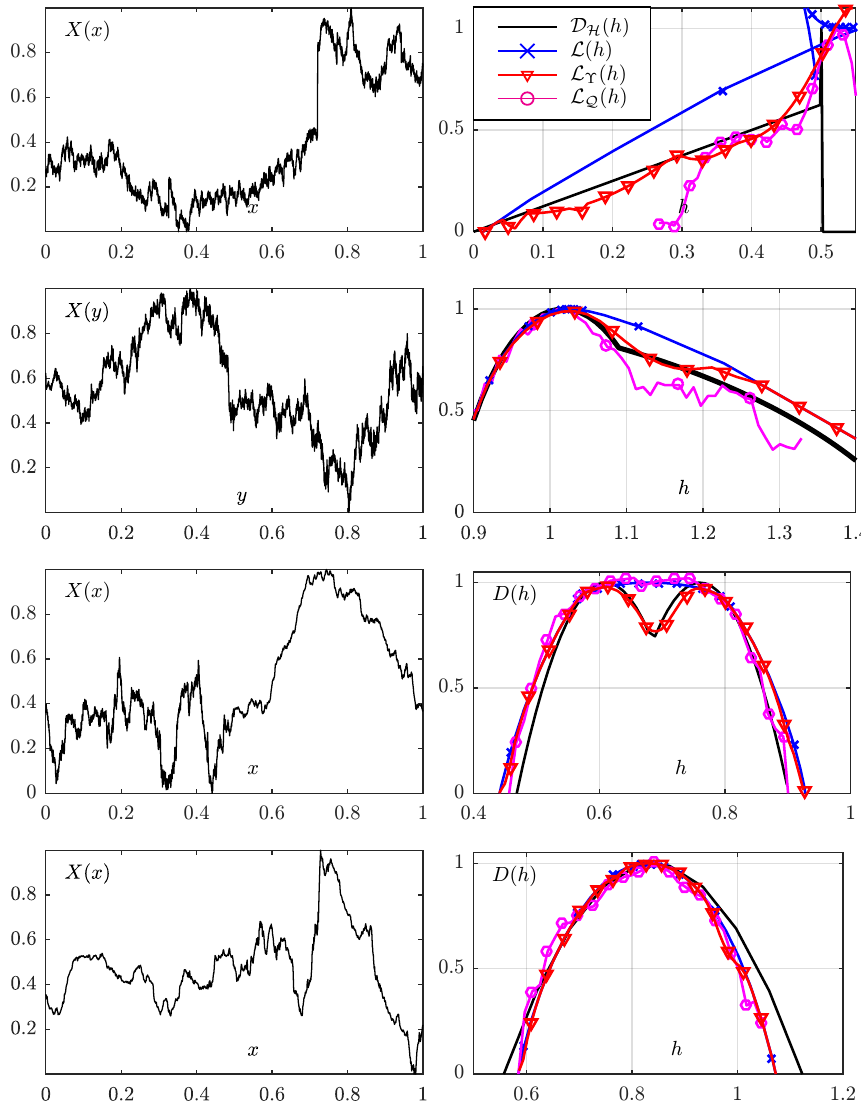}
\caption{\review{\textbf{1D multifractal processes: single realization spectra.} Sample realization (left
  column) and multifractal spectra (right column) \review{obtained from a single realization of (from top to bottom)} Levy process with Brownian component, thresholded
  binomial wavelet cascade, concatenated MRW, and homogeneous MRW.}}
\label{fig:results1d_single}
\end{figure}

\begin{figure}
  \centering
  \includegraphics[scale=1]{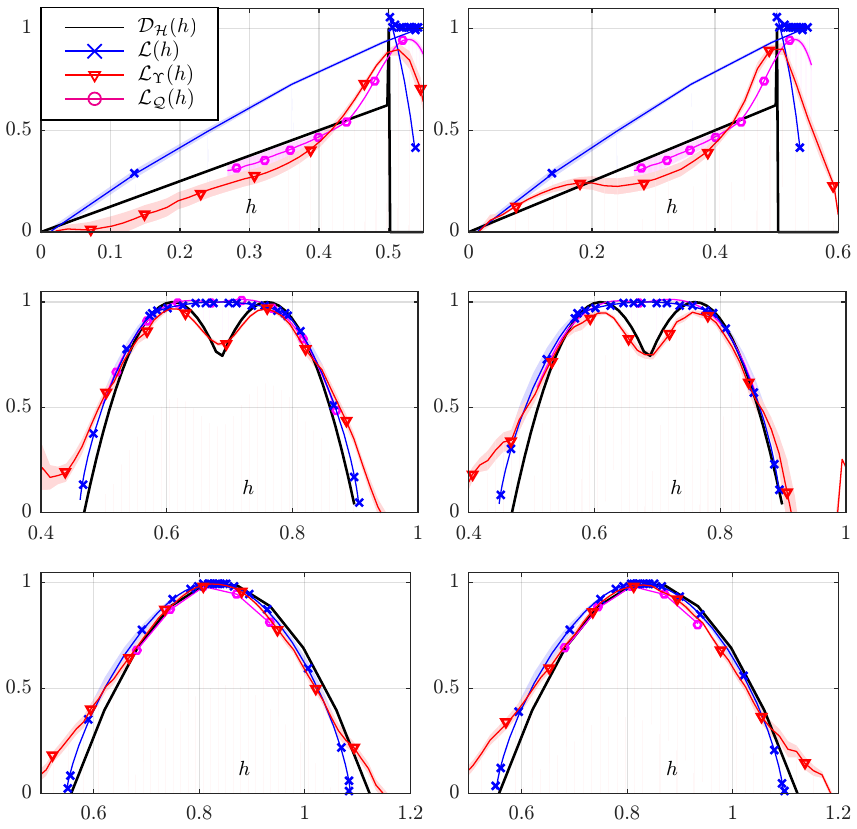}
\caption{\review{\textbf{1D multifractal processes: Monte Carlo-averaged spectra.} Estimated spectra with $g(h)=-\gamma h^2$ (left) and $g(h)=-\gamma |h|$, obtained as averages over 100 realizations, for (from top to bottom) Levy process with Brownian component, concatenated MRW, homogeneous MRW.
95\% confidence bands are reported using lightly shaded areas}.}
\label{fig:results1d_mc}
\end{figure}

\subsubsection{Varying function $g$}
\label{sec:varyg}
\review{Fig. \ref{fig:results1d_mc} (right column) further shows that estimated spectra for the three stochastic processes produced with the choice $g(h)=\gamma|h-\delta|$ are essentially equivalent to those obtained with the promoted choice $g(h)=\gamma (h-\delta)^2$ (left column).
This indicates that, though it may a priori appear as a critical issue, selecting the function $g$ turns out to have far less impact on the results than  expected.
It thus comforts the choice, proposed in \cite{costeniuc2005generalized} and
promoted here, of using a parabolic $g$ as an universal admissible function.}

\begin{figure}[t]
  \centering
  \includegraphics[scale=1]{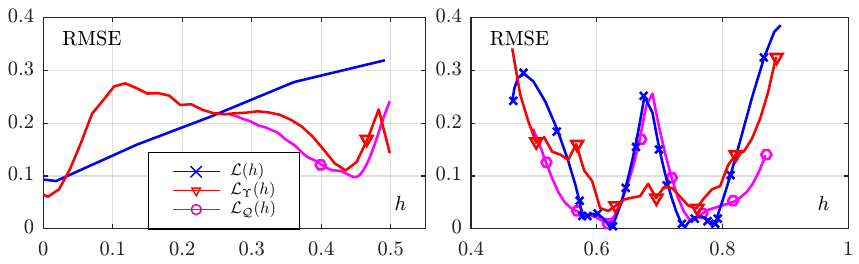}
\caption{\review{\textbf{1D multifractal processes: estimation performance.} Root mean-squared error
    (rmse) for the estimated multifractal spectra, for  $\alpha$-stable Lévy process with Brownian
    component (left), and concatenated MRW (right).}}
  \label{fig:results_1d_estperf}
\end{figure}

\subsubsection{2D processes}
\label{sec:2Dest}
We further show that the generalized multifractal spectrum approach is also operational for the analysis of images (i.e., $d=2$),
with extension to higher dimensions being straightforward and only requiring to substitute higher dimensional wavelet leaders (here, 2D leaders, cf., e.g., \cite{Wendt2009}) for the 1D leaders in \eqref{eq:def_phi}.
Fig.~\ref{fig:results2d} shows a synthetic image produced as the concatenation of two patches of MRW (left), its theoretical multifractal spectrum $\Dh$ and the corresponding estimates $\mathcal{L}$ and $\hg$ (right; averages over $100$ independent realizations).
Fig.~\ref{fig:results2d}  illustrates that $\dmin$ yields an excellent estimate for $\Dh$, and in particular a very satisfactory performance for the nonconcave region of the multifractal spectrum.
This example also suggests the use of estimators of nonconcave multifractal spectra as methods to detect nonstationarities.
It is indeed not obvious to the naked eye that the texture in Fig.~\ref{fig:results2d} (left) is formed by two patches with distinct regularity properties.
The information provided by $\dmin$
can thus shed light onto this situation and, for instance, suggest that segmentation and/or further analysis is needed.

\begin{figure}[t]
  \centering
  \includegraphics[scale=1]{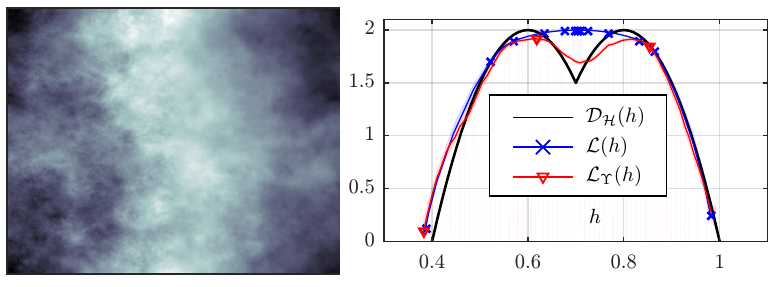}
\caption{\textbf{2D multifractal processes.} Sample realization (left column) and   multifractal
  spectra (right column; averages over $100$ realizations) for the horizontal concatenation of two
  rectangular 2D MRW.
\review{95\% confidence bands are reported in lightly shaded areas.
}
}
\label{fig:results2d}
\end{figure}

\subsubsection{Conclusions}
\label{sec}

Overall, this set of simulations unambiguously demonstrates that the proposed generalized multifractal
formalism i) can be used to estimate concave as well as nonconcave multifractal spectra, and ii) significantly improves the estimation over the state of the art estimates $\mathcal{L}$ and $\hq$, by being closer to the theoretical $\Dh$,
 confirming numerically that $\mathcal{L}(h)$ {\scriptsize$\gtrsim$} $\dmin(h)$ {\scriptsize$\gtrsim$} $\Dh(h)$, as expected from \eqref{eq:ineq}.

\section{Real world data}
We finally illustrate the application of the proposed generalized multifractal formalism to the analysis of a real-world satellite image corresponding to a patch of a hyperspectral image of the Moffett field, acquired by the AVIRIS instrument (spectral band 90), shown in Fig.~\ref{fig:results_moffet} (left).
The analysis was performed using parameters
$\gamma\in\{0,100,500,750\}$, $\delta\in[0.6,1.2]$, and $q\in[-10,10]$.
The estimates $\mathcal{L}$  and $\dmin$ are plotted in Fig.~\ref{fig:results_moffet} (right) and lead to the following conclusions.
First, the smooth, well-behaved and artifact-free shape of the estimated multifractal spectra and its large support suggest that a multifractal model can be reasonably assumed for the image.
Second, the Legendre spectrum $\mathcal{L}$ and generalized Legendre spectrum $\dmin$ coincide for a large part of their support, indicating that estimates are robust and stable.
Third, there is a clear difference between $\mathcal{L}(h)$  and $\dmin(h)$ for the interval
$h\in[0.98, 1.1]$: indeed, the proposed method allows to detect a nonconcave behavior of the
multifractal spectrum for this interval, leading to an estimate $\dmin$ with two modes that
are clearly separated by a pronounced, nonconcave dip.
In the absence of a ground truth, we cannot
conclude on the phenomenological origin of the observed nonconcave spectrum, although the example
studied in Section \ref{sec:2Dest} could suggest that the nonconcavity reflects the existence of two
visually discernable textures in the image (a rougher texture on the lower-right corner, and
a smoother one on the upper-left).
Nonetheless, we can conclude that the generalized multifractal
formalism enables the detection  of nonconcave multifractal spectrum  from real-world data, whatever
its precise origin, while it remains hidden when use is made of the conventional Legendre spectrum
$\mathcal{L}$.

\begin{figure}[t]
  \centering
  \includegraphics[scale=1]{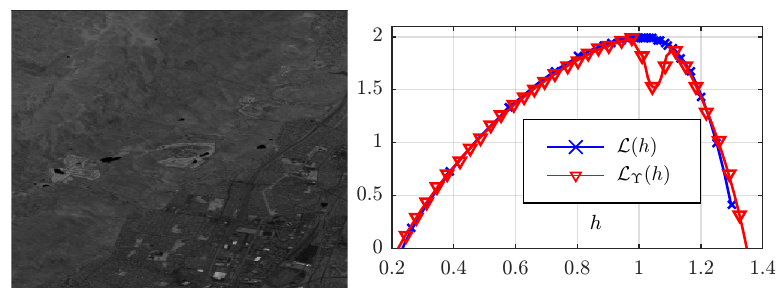}
\caption{\textbf{Real-world data.} Satellite image (left column) and multifractal spectrum estimates $\mathcal{L}(h)$ and $\dmin(h)$ (right column): The generalized spectrum $\dmin(h)$ clearly reveals the existence of a phenomenon leading to a nonconcave multifractal spectrum.
}
\label{fig:results_moffet}
\end{figure}

\section{Conclusions and Perspectives}
We proposed a generalized multifractal formalism that can be used to accurately
estimate nonconcave multifractal spectra. This formalism is based on a \review{generalized version} of the Legendre
transform  underlying the conventional multifractal formalism and leads to
numerically stable, yet not necessarily concave, estimates for the multifractal spectrum. This
\review{formalism} is inspired from generalized canonical ensembles studied in
statistical physics, and yields a multifractal formalism that inherits  numerical robustness while
capturing multifractal spectra of general shape. Moreover, we showed theoretically that this
generalized multifractal formalism leads to tighter bounds and is hence theoretically superior.
Based on these methodological developments, we devised a practical procedure for the estimation of nonconcave multifractal spectra from actual discrete, finite resolution data. The algorithm is operational and applicable to real-world data and will be made available.
The proposed methodology, theoretical results and practical algorithm have been  validated by
numerical simulations using both 1D and 2D synthetic multifractal processes, both with purely
concave multifractal spectra and with different types of nonconcavity. The results illustrate the
practical benefits of the proposed method and confirm numerically that it improves on
state-of-the-art techniques.
Future work will include the systematic analysis of real-world
biomedical signals and images, and the study of  recently introduced
second generation regularity exponents.

\appendix

\subsection{Proof of Proposition 1}
\label{sec:AppA}
  Denote by $\La_j$ the set of dyadic cubes of width $2^{-j}$. Let $\ep
  >0$ and $h$ be given. Let   $h'$  be such that $| h-h' |  \leq \ep$;
by definition of $\Dld$, the quantity
\(A_{\ep,j_n}=\textnormal{card} \{
  \la \in \La_{j_n} : 2^{-(h' + \ep ) j_n } \leq L_\la \leq 2^{-(h' - \ep ) j_n
  } \}
\)
satisfies
$ \forall \ep >0,  \exists j_n \rightarrow + \infty:\; A_{\ep,j_n} \geq 2^{(\D (h') - \ep) j_n}$.
Thus, for such  scales $j_n$, the quantity
\(
B_{\ep,j_n}=\textnormal{card} \{ \la \in \La_{j_n} : 2^{-(h +  2 \ep ) j_n } \leq L_\la \leq 2^{-(h
    - 2\ep ) j_n } \}
\)
satisfies
$  B_{\ep,j_n} \geq 2^{(\D (h') - \ep) j_n}$,
so that  $\D (h) \geq \D (h') - \ep$, hence the upper-semicontinuity.

\subsection{Proof of Proposition 2}
\label{sec:AppB}

The second inequality is given by Prop. 1. The first one can be written as
$f^{\star\star} + g^{\star\star} \geq (f+g)^{\star\star}$,
with $f=\Dld$ and $\mathcal{L}=f^{\star\star}$, noting that $g = g^{\star\star}$ (because $g$ is
concave).

Denote by $P_1$ the set of all affine functions. Then $f^{\star\star}$ can be written as
$  f^{\star\star} (x) = \inf \{ a(x): \; a \in P_1  \; \mbox{ and} \; a \geq f \}$
(see e.g. \cite{rockafellar1997convex,costeniuc2005generalized}).
Note that
$a \in P_1$  $a \geq f$, $b \in P_1$, $b \geq g$ implies that $a +b  \in P_1$  and $ a +b  \geq f  +g$.

 Therefore, if we define
 $A = \{   c = a+b \; \mbox{with}\; a, b \in P_1,   \;  a \geq f \; \mbox{and} \;  b \geq g\}$
 and
$ B = \{   c\in P_1 :  \; c \geq f + g\}$,
 it follows that that $A \subset B$.
 But
 $(f+g)^{\star\star} (x) =  \inf \{ c(x)\;  \mbox{with} \;  c \in B \}$ and
 $f^{\star\star}(x) + g^{\star\star}  (x) =  \inf \{c (x)  \; \mbox{with} \;  c \in A \}$.
 Since $A \subset B$, an infimum taken on $B$ is smaller than the same infimum on $A$; therefore
 $f^{\star\star} + g^{\star\star} \geq (f+g)^{\star\star}$, which completes the proof.

\subsection{Proof of Theorem~\ref{Theorem:global} }
\label{sec:AppC}

Without loss of generality, we can suppose that $g$ is nonpositive, and that, at its maximum, it takes the value 0. Let $D$ be a large deviation spectrum;
we first assume  that  $h_0$ is a point where $D(h_0) \geq 0 $.
Let $\ep >0$; since $D$ is upper-semicontinuous at $h_0$, $\exists \delta >0$  such that $ \forall h \in [h_0 - \delta, h_0 + \delta ] $, $D(h) \leq D(h_0) + \ep$. Note that one also has $\forall h, D(h) \leq d$.  It follows from the assumptions on $g$ that there exists a translation-dilation $g_\ep$ of $g$, such that
$ \forall h \notin  [h_0 -\ep, h_0 +\ep]$, $ g_\ep (h) \leq - d $ (and, by continuity, it is clear that the choice of $g_\ep $ can be restricted to a dense subset of the collection of all translates and dilates).
Thus, the function   $D(h) + g_\ep (h)$ is negative outside of $[h_0 -\ep, h_0 +\ep]$ and is less than $D(h_0) + g_\ep (h) + \ep $ in $[h_0 -\ep, h_0 +\ep]$.  Therefore it is everywhere less than $D(h_0) + \ep $; thus it is also the case  for its concave hull.  Consequently, the \review{generalized Legendre spectrum} of $D$ at $h_0$ is bounded by $D(h_0) + \ep$ (and it is also larger than $D(h_0)$). Since this is true $\forall \ep >0$, we can approximate arbitrarily well the value of $D$ at $h_0$.
 We now assume that $h_0$ is a point where $D(h_0)= -\infty $. By upper-semicontinuity, $D$ takes the value $-\infty$ in a neighborhood of $h_0$, and by picking  $g_\ep(h)=g( ah +b)$ with $a$ arbitrarily large, it is clear that the GLT of $D$ can take arbitrarily large negative values  at $h_0$.

\subsection{Proof of Theorem~\ref{theorem:gmff} }
\label{sec:AppD}

Let  $H$ be given and let $E_H = \{ x: \; h(x) =H\}$.   If $x \in E_H$, then
\eqref{eq:leader_decay} implies that there exists a sequence $\la_{j_n,k_x}$ of dyadic cubes such
that
 $x_0 \in \la_{j_n,k_x}$ and $-\log_2 (L_{\la_{j_n,k_x} }) / j_n \rightarrow h(x)$.  Let $J  >0$ be given; we pick a collection of maximal subcubes in the set
of cubes  $\la_{j_n,k_x}$ for $x \in E_H$, $j_n \geq J$ and $\left| - \log_2 (L_{\la_{j_n,k_x}}) / j_n - h(x) \right| \leq \ep $;
 and we denote by $\La_{j,H} $ the subcubes of this collection which are of width $2^{-j}$;  by construction,  $\bigcup_{j \geq J} \La_{j,H}$ is a covering of $E_H$.
 Restricting the sum in \eqref{sg} to the cubes $\la \in\La_{j,H} $  yields the  lower bound:
 \( \forall j,\; S_g (q,j)  \geq  2^{-dj} \sum_{ \lambda_{j',k'} \in \La_{j,H}}  (L_{j',k'})^q 2^{g(\phi_{j',k'}) j } . \)
 But, if $ \la_{j',k'} \in \La_{j,H}$, then  $ (L_{j',k'})^q \sim 2^{-Hqj}$ and $2^{g(\phi_{j',k'}) j } \sim 2^{g(H) j } $, so that
\begin{equation} \label{sg2}  \forall j, \qquad  S_g (q,j)  \geq  2^{-dj} \textnormal{card} ( \La_{j,H} ) 2^{-Hqj}   2^{g(H) j }  .  \end{equation}
Since  $\bigcup_{j \geq J} \La_{j,H}$ is a covering of $E_H$,  by definition of the Hausdorff dimension,
 \( \forall \ep >0,\; \sum_{ j \geq J} 2^{-j( { \mathcal D}_{\mathcal H}(H) -\ep ) } = + \infty  ,\)
so that there exists a sequence $j_l \rightarrow + \infty $ such that
\( \textnormal{card} ( \La_{j_l,H} )  2^{-j_l( { \mathcal D}_{\mathcal H}(H) -\ep ) }  \geq 1/ j_l^2 . \)
It follows from \eqref{sg2} that, for this sequence $j_l$,
 \[ S_g (q,j_l)  \geq  2^{-dj_l}  2^{-Hqj_l}   2^{g(H) j_l }  2^{-j_l( { \mathcal D}_{\mathcal H}(H) -\ep ) } / j_l^2 \]
 so that $ \zeta_g (q) = \liminf \log_2 ( S_g (q,j)) /(-j)  $
 satisfies
 \(\forall q, \forall H, \; \zeta_g (q)  \leq d +Hq -g(H) -{ \mathcal D}_{\mathcal H}(H)  \)
 and the result follows.

\bibliographystyle{IEEEtran}
\bibliography{biblio.bib}

\end{document}